\documentclass{amsart}

\usepackage[T1]{fontenc}
\usepackage[latin1]{inputenc}
\usepackage{amsmath}
\usepackage{amssymb}
\usepackage{amsthm}
\usepackage{dsfont}
\usepackage{color}

\theoremstyle{plain}\newtheorem{theo}{Theorem}
\theoremstyle{plain}
\theoremstyle{definition}
\theoremstyle{plain}\newtheorem{defi}{Definition}[section]
\theoremstyle{plain}
\theoremstyle{plain}\newtheorem{prop}[defi]{Proposition}
\theoremstyle{definition}

\newcommand{\N}{{\mathds N}}

\newcommand{\Z}{{\mathds Z}}
\DeclareMathOperator{\var}{Var}

\begin{document}

\title[$U$-statistic Processes Indexed by Random Walk]{Stable Limit Theorem for $U$-Statistic Processes Indexed by a Random Walk}


\author[B. Franke]{Brice Franke}
\address{Universit\'e de Brest, UMR CNRS 6205, 29238 Brest cedex, France}
\email{brice.franke@univ-brest.fr}
\author[F. P\`{e}ne]{Fran\c{c}oise P\`ene}
\address{Universit\'e de Brest, UMR CNRS 6205, 29238 Brest cedex, France}
\email{francoise.pene@univ-brest.fr}
\author[M. Wendler]{Martin Wendler}
\address{Ruhr-Universit\"at Bochum, 44780 Bochum, Germany}
\email{martin.wendler@rub.de}


\begin{abstract} Let $(S_n)_{n\in\N}$ be a $ \Z $-valued random walk with increments from the domain of attraction of some $\alpha$-stable law and let $(\xi(i))_{i\in\Z}$ be a sequence of iid random variables. 
We want to investigate $U$-statistics indexed by the random walk $S_n$, that is $U_n:=\sum_{1\leq i<j\leq n}h(\xi(S_i),\xi(S_j))$ for some symmetric bivariate function $h$. We will prove the weak convergence without assumption of finite variance. Additionally, under the assumption of finite moments of order greater than two, we will establish a law of the iterated logarithm for the $U$-statistic $U_n$.
\end{abstract}

\maketitle

\section{Introduction}
Random walks in random scenery were introduced by Kesten and Spitzer \cite{kest}. 
They  studied the partial sum process $\sum_{k=1}^n\xi(S_k)$, where $S_k:=\sum_{m=1}^kX_m$ is a random walk 
based on some sequence of $ \Z $-valued i.i.d. random variables  $(X_m)_{m\in\N}$ and  $(\xi(i))_{i\in\Z}$ is a sequence of 
real valued i.i.d. random variables which are supposed to be independent of the random walk $(S_k)_{k\in\N}$. 
The law of the random variable $ X_1 $ is supposed to belong to the normal domain of attraction of an $\alpha$-stable law
$F_\alpha$ with $0<\alpha\leq 2$, i.e.: one has
\begin{equation*}
P\left(n^{-\frac{1}{\alpha}}S_n\leq x\right)\rightarrow F_\alpha(x).
\end{equation*}
It is then well known that the sequence of stochastic processes
\begin{equation*}
S_t^{(n)}:=n^{-\frac{1}{\alpha}}S_{[nt]};\ \ t\geq0,n\in\N
\end{equation*}
converges in distribution towards an $\alpha$-stable L\'evy process $S_t^\star$ (see Skorokhod \cite{skor}, Theorem 2.7). 
It is further assumed that the random walk $(S_n)_{n\in\N}$ is irreducible and strongly aperiodic. For the case $ \alpha>1 $,
Kesten and Spitzer \cite{kest} showed that if the scenery variable $ \xi(1) $ is in the normal domain of attraction of some $ \beta $-stable law
$ F_\beta $ with $ 1<\beta\leq 2 $, then the partial sum process converges after some suitable renormalization toward
some specific continuous self-similar process $ (\Delta_t)_{t\geq0} $ with stationary increments. 
Since the random walk $ (S_k)_{k\in\N} $ can visit the same location several times, the sequence of random variables 
$ (\xi(S_k))_{k\in\N} $ shows some long range dependence. 
For the limit process $ (\Delta_t)_{t\geq0} $ this imposes some non-classical scaling index and also some non-trivial dependence of the stationnary increments. 
The construction of the process $ (\Delta_t)_{t\geq0} $ is given after Theorem \ref{theo1}.
The case $ \alpha\le 1 $ was studied extensively in \cite{Bolthausen,DU,cast}.  

A natural and widely applicable generalization of partial sums are $ U $-statistics. In this paper we want to investigate the
asymptotic behavior of $U$-statistics indexed by some random walk $ (S_k)_{k\in\N} $, which are defined as follows:
\begin{equation*}
U_n:=\sum_{1\leq i<j\leq n}h(\xi(S_i),\xi(S_j)).
\end{equation*}
In what follows we will assume that $h$ is a bivariate, measurable and symmetric function such that
\begin{equation} \label{Bedingungen}
E[|h(\xi(1),\xi(2))|]<\infty\ \ \text{and} \ \ E[|h(\xi(1),\xi(1))|]<\infty.
\end{equation}
We further assume that $Eh(\xi(1),\xi(2))=0$. The random walk $ (S_k)_{k\in\N} $ is supposed to fulfill the same assumptions as 
in the paper of Kesten and Spitzer \cite{kest} described above.
 
A classic approach for $U$-statistics in the case of finite second moments is the \hyphenation{Hoeff-ding} Hoeffding decomposition \cite{hoef}. 
We can write
\begin{equation*}
U_n=(n-1)\sum_{i=1}^nh_1(\xi(S_i))+\sum_{1\leq i<j\leq n}h_2(\xi(S_i),\xi(S_j))
\end{equation*}
with
\begin{align*}
h_1(x)&:=E(h(x,\xi(1)))\\
h_2(x,y)&:=h(x,y)-h_1(x)-h_1(y).
\end{align*}
We call $L_n:=\sum_{i=1}^nh_1(\xi(S_i))$ the linear part of the $U$-statistic $U_n$ and $R_n=\sum_{1\leq i<j\leq n}h_2(\xi(S_i),\xi(S_j))$ the remainder term. 
The idea of using the Hoeffding decomposition in the present context is based on the following crucial argument.
Under second moment assumptions, the family of random variables $(h_2(\xi_x,\xi_y))_{x\ne y}$ is pairwise uncorrelated. 
This leads to the following expression of the conditional variance of the remainder term, when the random walk is given
 \begin{equation*}
\var\left[R_n\big|(S_k)_{k\in\N}\right]=\sum_{x,y\in\Z}\var[h_2(\xi_x,\xi_y)]\left(\#\right\{(i,j)\big| i,j\leq n,\ i\neq j,\ S_i=x,\ S_j=y\left\}\right)^2.
\end{equation*}
By using moment bounds for the occupation times of the random walk $(S_n)_n$, that is the number of indices $i\le n$ such that $S_i=x$ (for $x\in \mathbb Z$), the 
asymptotic domination of the linear part $(n-1)L_n$ will follow.
In the particular case when $S_k=k$, our process $U_n$ corresponds to the 
classical $U$-statistic with independent entries $ \xi_i $. 
There under second moment assumptions on $\xi $ one has
$\var[R_n]=O(n^2)$ while $\var[(n-1)L_n]$ is of order $n^3$. It turns out that for the random walk in random scenery, considered here, the Hoeffding decomposition is helpful in spite of the dependence.

The $U$-statistic indexed by a random walk was examined by Cabus and Guillotin-Plantard \cite{cab} and Guillotin-Plantard and Ladret \cite{guil}, but only in the case of finite fourth moments.
 We will extend their results to the case when $h_1(\xi_1)$ is in the normal domain of attraction of a stable distribution. 
However, we then have to overcome another obstacle which results from the lack of second moments, which was crucial in the 
above argument.  
To deal with this problem, we will adapt a truncation method Heinrich and Wolf \cite{hein} introduced to study $U$-statistics with stable limits under independent random variables. 

Note that for other models of long range dependence (e.g. Gaussian sequences with slowly decaying covariances), both the linear part and the remainder term might contribute to the limit distribution, see Beutner and Z\"ahle \cite{beut}. Because of this, other methods, like representing the $U$-statistics as a functional of the empirical distribution function, are appropriate, see Dehling and Taqqu \cite{dehl}.

Another approach to study $ U $-statistics with heavy tails was presented by Da\-browski, Dehling, Mikosch, Sharipov \cite{dabr}.
They use a method based on functional convergence of suitably defined point processes to handle the existence of second moments.
This alternative approach was used in Franke, P\`ene and Wendler \cite{fpw2} to prove distributional convergence of 
$ U $-statistics indexed by a random walk. There it is proved that under some different assumptions on the scenery variables and
also under some different scaling the $ U $-statistics process converges toward some stochastic integral with respect to some 
L\'evy sheet.

We complete our study with a law of the iterated logarithm for the $U$-statistic process indexed by $S_n$, extending results from Lewis \cite{lewi} and Zhang \cite{zhan} for the partial sum indexed by a random walk.

\section{Main Results}

Our first theorem will establish the weak convergence of the $U$-statistic process without assuming that the summands of the linear part have second (or even higher) moments. More precisely, we will assume that
the law $\mathcal{L}(h_1(\xi(1)))$ is in the normal domain of attraction of a $\beta$-stable law $F_\beta$ with $1<\beta\leq 2$

For $1<\alpha\leq 2$ or $\beta=2$, the random walk in random scenery converges to a continuous limit (see \cite{kest,Bolthausen,DU}), even if the scenery contains jumps, so we define the continuous version of the $U$-statistics process
\begin{equation*}
U_{n}(t)=U_{k}\ \ \ \text{if} \ \ nt=k\in\N
\end{equation*}
and linear interpolated in between. We will prove weak convergence in the space of continuous functions $C[0,1]$ equipped with the supremum norm. For $\alpha\le 1$ and $\beta<2$, the limit process has a discontinuous limit (see Castell, Guillotin-Plantard, P\`ene \cite{cast}), so we will consider the space of c\`adl\`ag-functions $D[0,1]$ endowed with the Skorohod $M_1$-topology (see Skorokhod \cite{sko2}).

\begin{theo}\label{theo1}
Let $(S_n)_{n\in\mathbb Z}$ be a two-sided
random walk such that $S_0=0$ and with identically distributed increments $X_n=S_n-S_{n-1}$.
Assume that the law $\mathcal{L}(h_1(\xi(1)))$ is in the normal domain of attraction of a $\beta$-stable law $F_\beta$ with $1<\beta\leq 2$. 
Assume furthermore
that there exists  $\beta'>\beta$ such that $E|h(\xi(1),\xi(2))|^\eta<\infty$ with $\eta=\frac{2\beta'}{1+\beta'}$. 
Assume that the law $\mathcal{L}(X_1)$ is in the normal domain of attraction of an $\alpha$-stable law $F_\alpha$.
\begin{itemize}
\item If $1<\alpha\leq 2$, then we have the weak convergence in $C[0,1]$
\begin{equation*}
\left(n^{-2+\frac{1}{\alpha}-\frac{1}{\alpha\beta}}U_n(t)\right)_{t\in[0,1]}\Rightarrow(\Delta_t)_{t\in[0,1]},
\end{equation*}
with $\Delta_t$ as defined in Kesten and Spitzer \cite{kest}.
\item  If $\alpha=1$ 
and if $\frac{1}{n}\sum_{i=1}^n X_i$ converges in distribution to $aZ$ for a Cauchy-distributed random variable $Z$, then we have the following weak convergence in the sense of the finite
dimensional distributions\begin{equation*}
\left(n^{-1-\frac{1}{\beta}}(\log n)^{\frac{1-\beta}{\beta}}U_{\lfloor nt\rfloor}\right)_{t\in[0,1]}\Rightarrow\left(\frac{(\Gamma(\beta+1))^{\frac 1\beta}}{(a\pi)^{\frac{\beta-1}\beta}}Y_t\right)_{t\in[0,1]},
\end{equation*}
where $Y$ is the $\beta$-stable L\'evy-Process which is the limit of $n^{-1/\beta}\sum_{i=1}^{\lfloor nt\rfloor}\xi_i$. 
This weak convergence also holds in  $(D[0,1],M_1)$ if $1<\beta<2$
and in $(D[0,1],J_1)$ if $\beta=2$.
\item If $0<\alpha<1$, then we have the weak convergence in the sense of the finite dimensional distributions
\begin{equation*}
\left(n^{-1-\frac{1}{\beta}}U_{\lfloor nt\rfloor}\right)_{t\in[0,1]}\Rightarrow\left(bY_t\right)_{t\in[0,1]},
\end{equation*}
where $Y$ is the $\beta$-stable L\'evy-Process which is the limit of $n^{-1/\beta}\sum_{i=1}^{\lfloor nt\rfloor}\xi_i$ and
with $b=\left(E\left|\sum_{i\in\mathbb Z} \mathds{1}_{\{S_i=0\}}\right|
^{\beta-1}\right)^{1/\beta}$.
This weak convergence also holds in  $(D[0,1],M_1)$ if $1<\beta<2$
and in $(D[0,1],J_1)$ if $\beta=2$.
\end{itemize}
\end{theo}
Observe that we can choose $\beta'$ such that $\eta<\beta$, that means that the summands $h$ of the $U$-statistic might have less moments than $h_1$, so without loss of generality, we can assume that $E|h_1(\xi(1))|^\eta<\infty$.

To give the definition of the process $(\Delta_t)_{t\in[0,1]}$, we have to introduce some notation. Let $(T_t(x))_{t\geq0}$ be the local time of the limit process $(S_t^\star)_{t\geq0}$ of the rescaled partial sum $(n^{-\frac{1}{\alpha}}\sum_{i=1}^{[nt]}X_i)_{t\geq0}$, that means
\begin{equation*}
\int_0^t\mathds{1}_{[a,b)}(S_s^\star)ds=\int_a^bT_t(x)dx
\end{equation*}
almost surely. Let $(Z_{+}(t))_{t\geq0}$ and $(Z_{-}(t))_{t\geq0}$ be two independent copies of the limit process of the rescaled partial sum process $\left(n^{-\frac{1}{\beta}}\sum_{i=1}^{[nt]}h_1(\xi(i))\right)_{t\geq0}$. Then the limit process of the random walk in random scenery is defined as
\begin{equation*}
\Delta_t=\int_0^\infty T_t(x)dZ_+(x)+\int_0^\infty T_t(-x)dZ_-(x).
\end{equation*}

For random walks in random scenery, Lewis \cite{lewi} and Khoshnevisan and Lewis \cite{khos} proved the law of the iterated logarithm. This was improved by Cs\'aki,  K\"onig, Shi \cite{csak} and Zhang \cite{zhan} using strong approximation methods. In our second theorem, we will extend these results to $U$-statistics:

\begin{theo}\label{theo2}
Let the assumption of Theorem 1 hold with $\alpha=\beta=2$ and additional $E|h_1(\xi(i))|^p<\infty$ and $E|X_i|^p<\infty$ for some $p>2$. Then
\begin{align*}
\limsup_{n\rightarrow\infty}\frac{U_n}{n^{\frac{7}{4}}(\log\log n)^{\frac{3}{4}}}&=\frac{2^{\frac{1}{4}}\var(h_1(\xi(1)))^{\frac{1}{2}}}{3\var(X)^{\frac{1}{4}}}\\
\liminf_{n\rightarrow\infty}\frac{U_n}{n^{\frac{7}{4}}(\log\log n)^{\frac{3}{4}}}&=-\frac{2^{\frac{1}{4}}\var(h_1(\xi(1)))^{\frac{1}{2}}}{3\var(X)^{\frac{1}{4}}}\\
\end{align*}
almost surely.
\end{theo}

\section{Auxiliary Results}

We define the occupation times $N_n(x):=\sum_{i=1}^n\mathds{1}_{\{S_i=x\}}$. Let us write $V_n:=\sum_{x\in\Z}N_n^2(x)$. Observe that
\begin{equation*}
V_n=\sum_{x\in\Z}(\sum_{k=1}^n \mathds{1}_{\{S_k=x\}})^2=\sum_{x\in\Z}\sum_{k,l=1}^n \mathds{1}{\{S_k=S_l=x\}}=\sum_{k,l=1}^n\mathds{1}_{\{S_k=S_l\}}
\end{equation*}
equals $n$ plus the number of self intersections of the random walk.

\begin{prop}\label{lem1}
If $0<\alpha\leq 2$, then
\begin{align*}
E\left[V_n\right]&=O(n^{2-\frac{1}{\alpha'}}\log n),\\
E\left[V_n^2\right]&=O(n^{4-\frac{2}{\alpha'}}\log^2n),
\end{align*}
with $\alpha':=\max(1,\alpha)$.
\end{prop}
For the case $1<\alpha\leq 2$, this follows from Lemma 2.1 of Guillotin-Plantard, Ladret \cite{guil}. For the case $\alpha=1$ and more precise results, we refer to Deligiannidis, Utev \cite{DU}. The bound for the case $0<\alpha<1$ comes for example from the proof of Lemma 19 in Castell, Guillotin-Plantard, P\`ene,  Schapira \cite{BFFN} in which it is proven that there exists $\gamma>0$ such that $\sup_n{\mathbb E}[\exp(\gamma V_n/n)]<\infty$.

\begin{prop}\label{lem2} Under the conditions of Theorem 1 for $0<\alpha\leq 2$, we have that
\begin{equation*}
\max_{k\leq n}R_k=o(n^{2-\frac{1}{\alpha'}+\frac{1}{\alpha'\beta}})
\end{equation*}
almost surely with $\alpha':=\max(1,\alpha)$.
\end{prop}

\begin{proof}
We define $a_l=2^{l\frac{1+\beta'}{\alpha'\beta'}}$  and the truncated kernel
\begin{equation*}
h_{0,l}(x,y):=\begin{cases}h(x,y) \ \ &\text{if} \ |h(x,y)|\leq a_l\\0&\text{else}\end{cases}.
\end{equation*}
We also need the Hoeffding decomposition for the truncated kernel:
\begin{align*}
h_{1,l}(x)&:=E(h_{0,l}(x,\xi(1)))\\
h_{2,l}(x,y)&:=h_{0,l}(x,y)-h_{1,l}(x)-h_{1,l}(y).
\end{align*}
We introduce the following notation:
\begin{align*}
\tilde{L}_{l,n}&:=\sum_{i=1}^nh_{1,l}(\xi(S_i))\\
\tilde{U}_{l,n}&:=\sum_{1\leq i<j\leq n}h_{0,l}(\xi(S_i),\xi(S_j))\\
\tilde{R}_{l,n}&:=\sum_{1\leq i<j\leq n}h_{2,l}(\xi(S_i),\xi(S_j)).
\end{align*}
Recall the Hoeffding decomposition
\begin{equation*}
U_n=(n-1)L_n+R_n.
\end{equation*}
Similar, we have that 
\begin{equation*}
\tilde{U}_{l,n}=(n-1)\tilde{L}_{l,n}+\tilde{R}_{l,n}.
\end{equation*}
We now obtain the following representation for the remainder term:
\begin{multline*}
R_n=U_n-(n-1)L_n=(U_n-\tilde{U}_{l,n})-(n-1)L_n+\tilde{U}_{l,n}\\
=(U_n-\tilde{U}_{l,n})-(n-1)L_n+(n-1)\tilde{L}_{l,n}+\tilde{R}_{l,n}=(U_n-\tilde{U}_{l,n})-(n-1)(L_n-\tilde{L}_{l,n})+\tilde{R}_{l,n}
\end{multline*}
We will treat the three summands separately. To do this we have to show
\begin{enumerate}
\item \begin{equation*}
\max_{n\leq 2^l}|U_n-\tilde{U}_{l,n}|=o(2^{l(2-\frac{1}{\alpha'}+\frac{1}{\alpha'\beta})})\quad \mbox{almost surely},
\end{equation*}
\item \begin{equation*}
\max_{n\leq 2^l}2^l|L_n-\tilde{L}_{l,n}|=o(2^{l(2-\frac{1}{\alpha'}+\frac{1}{\alpha'\beta})})\quad \mbox{almost surely},
\end{equation*}
\item \begin{equation*}
\max_{n\leq 2^l}|\tilde{R}_{l,n}|=o(2^{l(2-\frac{1}{\alpha'}+\frac{1}{\alpha'\beta})})\quad \mbox{almost surely}.
\end{equation*}
\end{enumerate}
In the proof of (1), we have to deal with the problem that we might have $S_i=S_j$ for $i\neq j$. For this reason we will treat separately the cases $ S_i=S_j $ and $ S_i\neq S_j $:
\begin{multline*}
|U_n-\tilde{U}_{l,n}|\leq \left|\sum_{\substack{1\leq i<j\leq n\\S_i\neq S_j}}(h(\xi(S_i),\xi(S_j))-h_{0,l}(\xi(S_i),\xi(S_j)))\right|\\
+\left|\sum_{\substack{1\leq i<j\leq n\\S_i=S_j}}(h(\xi(S_i),\xi(S_j))-h_{0,l}(\xi(S_i),\xi(S_j)))\right|=|A_{l,n}|+|B_{l,n}|.
\end{multline*}
In order to establish bounds for the maximum, we have to control the increments of $A_{l,n}$. Let $n_1,n_2\in\N$ with $n_1\leq n_2\leq 2^l$, then
\begin{equation*}
A_{l,n_2}-A_{l,n_1}=\sum_{\substack{1\leq i<j\leq n\\ n_1<j\leq n_2 \\S_i\neq S_j}}(h(\xi(S_i),\xi(S_j))-h_{0,l}(\xi(S_i),\xi(S_j))),
\end{equation*}
so we have at most $2^l(n_2-n_1)$ summands of $A_{l,n_2}-A_{l,n_1}$ and for every summand
\begin{eqnarray*}
E\left|h(\xi(S_i),\xi(S_j))-h_{0,l}(\xi(S_i),\xi(S_j))\right|&=&E|h(\xi(1),\xi(2))\mathds{1}_{\{|h(\xi(1),\xi(2))|>a_l\}}|\\
    &\leq& a_l^{1-\eta}E|h(\xi(1),\xi(2))\mathds{1}_{\{|h(\xi(1),\xi(2))|>a_l\}}|^{\eta}\\
     &\leq& a_l^{1-\eta}E|h(\xi(1),\xi(2))|^{\eta}.
\end{eqnarray*}
Consequently, we have by the triangular inequality that
\begin{eqnarray*}
E|A_{l,n_2}-A_{l,n_1}|&\leq& 2^l(n_2-n_1)E|h(\xi(1),\xi(2))\mathds{1}_{\{|h(\xi(1),\xi(2))|>a_l\}}|\\
                                          &\leq& 2^l(n_2-n_1)a_l^{1-\eta}E|h(\xi(1),\xi(2))|^{\eta}
\end{eqnarray*}
Fix some $\theta$ such that $1<\theta	<1+\frac 1{\alpha'\beta}-\frac 1{\alpha'\beta'}$
(such a $\theta$ exists since $\beta'>\beta$).
Observe that
$$
E|A_{l,n_2}-A_{l,n_1}|\leq C_0 2^l(n_2-n_1)^\theta a_l^{1-\eta}$$
with the constant $C_0=E|h(\xi(1),\xi(2))|^{\eta}$.
We can write $A_{l,n}=\sum_{i=1}^n(A_{l,i}-A_{l,i-1})$ (with $A_{l,0}:=0$) and in the same way $A_{l,n_2}-A_{l,n_1}=\sum_{i=n_1+1}^{n_2}(A_{l,i}-A_{l,i-1})$, so due to the maximal
 inequality given in Theorem 1 of M\'oricz \cite{mori} (applied with $\gamma=1$, $\theta$ instead of $\alpha$ and $(g(F_{b,n}))^\theta= C_0 2^ln^\theta a_l^{1-\eta}$), we obtain
\begin{equation*}
E\left|\max_{n\leq 2^l}A_{l,n}\right|\leq C_02^{l(1+\theta)}a_l^{1-\eta}.
\end{equation*}
Recall that $a_l=2^{l\frac{1+\beta'}{\alpha'\beta'}}$ and $\eta=\frac{2\beta'}{1+\beta'}$, so $1-\eta=\frac{1-\beta'}{1+\beta'}$. It follows from the Markov inequality that
\begin{multline*}
\sum_{l=1}^\infty P\left(\frac{1}{2^{l(2-\frac{1}{\alpha'}
    +\frac{1}{\alpha'\beta})}}\max_{n\leq 2^l}A_{l,n}\geq \epsilon\right)\leq \frac{C_0}{\epsilon}\sum_{l=1}^\infty\frac{2^{l(1+\theta)}a_l^{1-\eta}}{2^{l(2-\frac{1}{\alpha'}+\frac{1}{\alpha'\beta})}}\\
=\frac{C_0}{\epsilon}\sum_{l=1}^\infty\frac{2^{l\frac{1+\beta'}{\alpha'\beta'}\frac{1-\beta'}{1+\beta'}}}{2^{l(1-\theta-\frac{1}{\alpha'}
+\frac{1}{\alpha'\beta})}}=\frac{C_0}{\epsilon}\sum_{l=1}^\infty\frac{2^{l(\frac{1}{\alpha'\beta'}-\frac{1}{\alpha'})}}{2^{l(1-\theta-\frac{1}{\alpha'}+\frac{1}{\alpha'\beta})}}
=\frac{C_0}{\epsilon}\sum_{l=1}^\infty 2^{l(\frac{1}{\alpha'\beta'}-\frac{1}{\alpha'\beta}+\theta-1)}<\infty,
\end{multline*}
as $\frac{1}{\alpha'\beta'}-\frac{1}{\alpha'\beta}+\theta-1<0$. With the Borel-Cantelli lemma, we can now conclude that
\begin{equation*}
P\left(\frac{1}{2^{l(2-\frac{1}{\alpha'}+\frac{1}{\alpha'\beta})}}\max_{n\leq 2^l}A_{l,n}\geq \epsilon\text{ infinitely often }\right)=0
\end{equation*}
and thus $\max_{n\leq 2^l}A_{l,n}=o(2^{l(2-\frac{1}{\alpha'}+\frac{1}{\alpha'\beta})})$ almost surely. For $B_{l,n}$, we use the fact that the sequences $(S_n)_{n\in\N}$ and $(\xi(n))_{n\in\N}$ 
are independent and observe that
\begin{eqnarray*}
E|\max_{n\leq 2^l}B_{l,n}|&\leq& E\max_{n\leq 2^l}\sum_{\substack{1\leq i<j\leq n\\S_i=S_j}}\left|(h(\xi(S_i),\xi(S_j))-h_{0,l}(\xi(S_i),\xi(S_j)))\right|\\
&\leq& E\sum_{\substack{1\leq i<j\leq 2^l\\S_i=S_j}}\left|(h(\xi(S_i),\xi(S_j))-h_{0,l}(\xi(S_i),\xi(S_j)))\right|\\
&\leq& E\#\left\{(i,j)\ :\ 1\leq i<j\leq 2^l|S_i=S_j)\right\}E|h(\xi(1),\xi(1))\mathds{1}_{\{|h(\xi(1),\xi(1))|>a_l\}}|\\
&\leq& E\left|\sum_{x\in\Z}N_{2^l}^2(x)\right|E\left|h(\xi(1),\xi(1))|\leq C2^{l(2-\frac{1}{\alpha'})}\right|
\end{eqnarray*}
for some constant $C$, where we used Proposition \ref{lem1} for the occupation times $N_n(x):=\sum_{i=1}^n\mathds{1}_{\{S_i=x\}}$. Again using the Markov inequality, we arrive at
\begin{equation*}
\sum_{l=1}^\infty P\left(\frac{1}{2^{l(2-\frac{1}{\alpha'}+\frac{1}{\alpha'\beta})}}\max_{n\leq 2^l}B_{l,n}\geq \epsilon\right)\leq \frac{C}{\epsilon}\sum_{l=1}^\infty\frac{2^{l(2-\frac{1}{\alpha'})}l}{2^{l(2-\frac{1}{\alpha'}+\frac{1}{\alpha'\beta})}}\leq \frac{C}{\epsilon}\sum_{l=1}^\infty2^{-\frac{1}{\alpha'\beta}l}l<\infty
\end{equation*}
and,  as above, the Borel-Cantelli lemma leads to $\max_{n\leq 2^l}B_{l,n}=o(2^{l(2-\frac{1}{\alpha'}+\frac{1}{\alpha'\beta})})$ almost surely, which completes the proof of (1). To prove (2),  note that
\begin{eqnarray*}
E|h_1(\xi(1))-h_{1,l}(\xi(1))|&=&E|E[h(\xi(1),\xi(2))|\xi(1)]-E[h(\xi(1),\xi(2))\mathds{1}_{\{|h(\xi(1),\xi(2))|\leq a_l\}}|\xi(1)]|\\
  &=&E|E[h(\xi(1),\xi(2))\mathds{1}_{\{|h(\xi(1),\xi(2))|> a_l\}}|\xi(1)]|\\
   &\leq& E\left[E\left[|h(\xi(1),\xi(2))\mathds{1}_{\{|h(\xi(1),\xi(2))|> a_l\}}|\big|\xi(1)\right]\right]\\
    &=& E\left|h(\xi(1),\xi(2))\mathds{1}_{\{|h(\xi(1),\xi(2))|> a_l\}}\right|\leq \frac{1}{a_l^{\eta-1}}E|h(\xi(1),\xi(2))|^\eta.
\end{eqnarray*}
With the triangular inequality and the assumption that $E|h(\xi(1),\xi(2))|^\eta<\infty$, it follows that for some constant $C$ and any $n_1,n_2\in\N$ with $n_1\leq n_2$
\begin{equation*}
E|\sum_{i=n_1+1}^{n_2}(h_1(\xi(S_i))-h_{1,l}(\xi(S_i)))|\leq C(n_2-n_1)a_l^{1-\eta}.
\end{equation*}
Again, we apply the maximal inequality in Theorem 1 of M\'oricz \cite{mori} and obtain
\begin{equation*}
E\max_{n\leq 2^l}2^l|L_n-\tilde{L}_{l,n}|\leq C2^{l(1+\theta)}a_l^{1-\eta}
\end{equation*}
for every $\theta>1$ and we can proceed in the same way as we proved almost sure convergence for $A_{l,n}$. So it remains to show the last part (3). We will prove that
\begin{align*}
\max_{n\leq 2^l}|E\tilde{R}_{l,n}|&=o(2^{l(2-\frac{1}{\alpha'}+\frac{1}{\alpha'\beta})})\\
\max_{n\leq 2^l}|\tilde{R}_{l,n}-E\tilde{R}_{l,n}|&=o(2^{l(2-\frac{1}{\alpha'}+\frac{1}{\alpha'\beta})})\quad\mbox{almost surely}.
\end{align*}
We obtain with a short calculation that
\begin{equation*}
\tilde{R}_{l,n}=\tilde{U}_{l,n}-(n-1)\tilde{L}_{l,n}=(\tilde{U}_{l,n}-U_n)+(n-1)(L_n-\tilde{L}_{l,n})+R_n
\end{equation*}
and consequently
\begin{eqnarray*}
\max_{n\leq 2^l}|E\tilde{R}_{l,n}| &=& \max_{n\leq 2^l}\left|E(\tilde{U}_{l,n}-U_n)+(n-1)E(L_n-\tilde{L}_{l,n})+ER_n\right|\\
       &\leq& E\max_{n\leq 2^l}|U_n-\tilde{U}_{l,n}|+(n-1)E\max_{n\leq 2^l}|L_n-\tilde{L}_{l,n}|+\max_{n\leq 2^l}|ER_n|.
\end{eqnarray*}
We have already shown in (1) and (2) that the first two summands are of order $o(2^{l(2-\frac{1}{\alpha'}+\frac{1}{\alpha'\beta})})$. For the last summand, we use the fact that 
$$Eh_2(\xi(1),\xi(2))=Eh(\xi(1),\xi(2))-Eh_1(\xi(1))-Eh_1(\xi(2))=0$$ to see that
only the indices with $S_i=S_j$ contribute to the expectation 
\begin{equation*}
ER_n=E\left[\sum_{\substack{1\leq i<j\leq n\\ S_i=S_j}}h_2(\xi(S_i),\xi(S_j))\right]
\end{equation*}
and due to \ref{Bedingungen} and Proposition \ref{lem1} we have
\begin{eqnarray*}
\max_{n\leq 2^l}|ER_n|&\leq& \max_{n\leq 2^l}E\#\left\{(i,j)\ :\ 1\leq i<j\leq n|S_i=S_j\right\}E|h_2(\xi(1),\xi(1))|\\
   &\leq& E\#\left\{(i,j)\ :\ 1\leq i<j\leq 2^l|S_i=S_j\right\}E|h_2(\xi(1),\xi(1))|\\
  &\leq& C2^{l(2-\frac{1}{\alpha'})}l=o(2^{l(2-\frac{1}{\alpha'}+\frac{1}{\alpha'\beta})}).
\end{eqnarray*}
To show the convergence of the remaining part, we first decompose it as
\begin{eqnarray*}
\tilde{R}_{l,n}-E\tilde{R}_{l,n}&=&\sum_{\substack{1\leq i<j\leq n\\S_i\neq S_j}}(h_{2,l}(\xi(S_i),\xi(S_j))-Eh_{2,l}(\xi(1),\xi(2)))\\
      &&+\sum_{\substack{1\leq i<j\leq n\\S_i=S_j}}(h_{2,l}(\xi(S_i),\xi(S_j))-Eh_{2,l}(\xi(1),\xi(1)))=: C_{l,n}+D_{l,n}.
\end{eqnarray*}
For $D_{l,n}$, we have by the independence of $(S_n)_{n\in\N}$ and $(\xi(n))_{n\in\N}$ and the fact
\begin{eqnarray*}
E|h_{2,l}(\xi(1),\xi(1))|&\leq& E|h_{0,l}(\xi(1),\xi(1))|+2E|h_{1,l}(\xi(1))|\\
&\leq& E|h(\xi(1),\xi(1))|+2E|h(\xi(1),\xi(2))| \ < \ \infty,
\end{eqnarray*}
that
\begin{equation*}
E\max_{n\leq 2^l}|D_{l,n}|\leq E\#\left\{(i,j)\ :\ 1\leq i<j\leq 2^l|S_i=S_j\right\}2E|h_{2,l}(\xi(1),\xi(1))|\leq C2^{l(2-\frac{1}{\alpha'})}l.
\end{equation*}
In the same way as for $B_{l,n}$ we now can conclude that $\max_{n\leq 2^l}D_{l,n}=o(2^{l(2-\frac{1}{\alpha'}+\frac{1}{\alpha'\beta})})$ almost surely. Finally, we will deal with $C_{l,n}$. Recall that $h_{0,l}$ is bounded by $a_l$, so $h_{2,l}$ is bounded by $3a_l$. By the triangular inequality for the $L_\eta$-norm, we have that
\begin{eqnarray*}
E|h_{2,l}(\xi(1),\xi(2))|^\eta&\leq& \left( \left\|h_{0,l}(\xi(1),\xi(2))\right\|_\eta+2\left||h_{1,l}(\xi(1))\right\|_\eta\right)^\eta\\
&\leq& \left(3\left\|h(\xi(1),\xi(2))\right\|_\eta\right)^\eta,
\end{eqnarray*}
and as a consequence for some constant $C>0$
\begin{equation*}
E\left(h_{2,l}(\xi(1),\xi(2))\right)^2\leq (3a_l)^{2-\eta}E|h_{2,l}(\xi(1),\xi(2))|^\eta\leq C2^{l\frac{1+\beta'}{\alpha'\beta'}(2-\frac{2\beta'}{1+\beta'})}=C2^{l\frac{2}{\alpha'\beta'}}.
\end{equation*} 
Furthermore we have the property of the Hoeffding decomposition that the random variables $h_{2,l}(\xi(1),\xi(2))$ and $h_{2,l}(\xi(1),\xi(3))$ are uncorrelated, see Lee \cite{lee}, page 30. So we can find bounds for the conditional variance of the increments of $C_{l,n}$. To simplify the notation, we write
\begin{equation*}
Y(i,j):=h_{2,l}(\xi(i),\xi(j))\mathds{1}_{\{i\neq j\}}-\left(Eh_{2,l}(\xi(i),\xi(j))\right)\mathds{1}_{\{i\neq j\}}
\end{equation*}
and obtain for $n_1\leq n_2\leq 2^l$
\begin{eqnarray*}
&& E\left[\left(C_{l,n_2}-C_{l,n_1}\right)^2\big|(X_k)_{k\in\N}\right]=E\left[\left(\sum_{\substack{1\leq i<j\leq n_2\\n_1< j\leq n_2}}Y(S_i,S_j)\right)^2\big|(X_k)_{k\in\N}\right]\\
&=&\sum_{\substack{x,y\in\Z\\x<y}}\Big(\#\left\{(i,j)\ :\ 1\leq i< j\leq n_2|n_1<j\leq n_2,\ S_i=x,\ S_j=y\right\}\\
&&+\#\left\{(i,j)\ :\ 1\leq i< j\leq n_2|n_1<j\leq n_2,\ S_i=y,\ S_j=x\right\}\Big)^2E(Y(x,y))^2\\
&\leq& C2^{l\frac{2}{\alpha'\beta'}}2\sum_{x,y\in\Z}(N_{n_2}(x)(N_{n_2}(y)-N_{n_1}(y)))^2.
\end{eqnarray*}
Due to Theorem 3 of M\'oricz \cite{mori} apllied with $ \gamma=2 $ and the (random) superadditive function
\begin{equation*}
g(F_{b,n})=C2^{\ell\frac2{\alpha'\beta'}+1}\sum_{x,y\in\Z}N_{2^\ell}(x)^2(N_{b+n}(y)-N_{b}(y))^2.
\end{equation*}
It follows that
\begin{equation*}
E\left[\max_{n\leq 2^l}\left(\sum_{1\leq i<j\leq n}Y(S_i,S_j)\right)^2\big|(X_k)_{k\in\N}\right]\leq C2^{l\frac{2}{\alpha'\beta'}}\sum_{x\in\Z}N_{2^l}^2(x)\sum_{y\in\Z}N_{2^l}^2(y)l^2.
\end{equation*}
Taking the expectation with respect to $(X_k)_{k\in\N}$, we get the following bound using Proposition \ref{lem1} at
\begin{eqnarray*}
E\left(\max_{n\leq 2^l}C_{l,n}\right)^2&=&E\left[E\left[\max_{n\leq 2^l}\left(\sum_{1\leq i<j\leq n}Y(S_i,S_j)\right)^2\big|(X_k)_{k\in\N}\right]\right]\\
 &\leq& C2^{l\frac{2}{\alpha'\beta'}}l^2E\left(\sum_{x\in\Z}N_{2^l}^2(x)\right)^2\\
   &\leq& C2^{l\frac{2}{\alpha'\beta'}}l^42^{l(4-\frac{2}{\alpha'})}\\
  &=&C2^{l2(2-\frac{1}{\alpha'}+\frac{1}{\alpha'\beta'})}l^4.
\end{eqnarray*}
We can now use the Chebyshev inequality and arrive at
\begin{eqnarray*}
\sum_{l=1}^\infty P\left(\frac{1}{2^{l(2-\frac{1}{\alpha'}+\frac{1}{\alpha'\beta})}}\max_{n\leq 2^l}C_{l,n}\geq \epsilon\right)&\leq& \frac{C}{\epsilon^2}\sum_{l=1}^\infty\frac{2^{l2(2-\frac{1}{\alpha'}+\frac{1}{\alpha'\beta'})}l^4}{2^{l2(2-\frac{1}{\alpha'}+\frac{1}{\alpha'\beta})}}\\
&\leq&  \frac{C}{\epsilon^2}\sum_{l=1}^\infty 2^{l2(\frac{1}{\alpha'\beta'}-\frac{1}{\alpha'\beta})} l^4 \ < \ \infty
\end{eqnarray*}
and the Borel-Cantelli lemma completes the proof.
\end{proof}

\section{Proofs of Main Results}

\begin{proof}[Proof of Theorem \ref{theo1}] Recall the Hoeffding decomposition
\begin{equation*}
n^{-2+\frac{1}{\alpha'}-\frac{1}{\alpha'\beta}}U_{[nt]}=\frac{n-1}{n}n^{-1+\frac{1}{\alpha'}-\frac{1}{\alpha'\beta}}L_{[nt]}+n^{-2+\frac{1}{\alpha'}-\frac{1}{\alpha'\beta}}R_{[nt]}.
\end{equation*}
For the linear part in the case $1<\alpha\leq 2$, we apply Theorem 1.1 of Kesten and Spitzer \cite{kest} to the random variables $h_1(\xi(i))$ and conclude that $\left(n^{-1+\frac{1}{\alpha}-\frac{1}{\alpha\beta}}L_{[nt]}\right)_{t\in[0,1]}$ converges weakly to $\left(\Delta_t\right)_{t\in[0,1]}$. 

In the case
$\alpha=1$ and $\beta=2$, due to Theorem 2.3 in \cite{Bolthausen,DU},
$\left(L_{[nt]}/\sqrt{n\log n}\right)_{t\in[0,1]}$ converges in distribution to $\left(\frac{\sqrt{2}}{\sqrt{a\pi }}Y_t\right)_{t\in[0,1]}$ as $n$ goes to infinity.

In the case
$\alpha=1$ and $1<\beta< 2$, due to Theorem 1 of \cite{cast},  
$\left(L_{[nt]}/({n^{\frac 1\beta}(\log n)^{\frac{\beta-1}\beta}})\right)_{t\in[0,1]}$ converges in distribution (with respect to the $M_1$-metric) to $\left(\frac{(\Gamma(\beta+1))^{\frac 1 \beta}}{(a\pi )^{\frac{\beta-1}\beta}}Y_t\right)_{t\in[0,1]}$ as $n$ goes to infinity.

In the case $\alpha<1$ and $1<\beta\le 2$, the convergence of 
$\left(L_{[nt]}/n^{\frac 1\beta}\right)_{t\in[0,1]}$ is proved namely in Remark 2 of \cite{cast}.

For the tightness in $(D([0,1],J_1)$ when $\alpha<1$ and $\beta=2$,
we follow the proof of tightness Bolthausen in \cite{Bolthausen}. 
For completeness, we explain the adaptations to make.
Using the fact that $L_n(0)\equiv 0$ and that $(L_i)_{i\ge 0}$ is a sequence of partial sums of a stationary sequence, it is enough to prove that, for every $\varepsilon>0$,
there exists $\lambda>0$ such that
\begin{equation}\label{eq:tightness}
\exists n_0\ge 1,\quad \forall n\ge n_0,\quad \lambda^2\mathbb P\left(\max_{j\le n}|L_j|>\lambda\sqrt{ n}\right)\le \varepsilon.
\end{equation}
Recall that $(V_n/n)_n$ converges almost surely to $b^2$ (see for example \cite[p. 10]{kest}) and write
$$\mathbb P\left(\max_{j\le n}|L_j|>\lambda\sqrt{ n}\right)\le   
       \mathbb P\left(\max_{j\le n}|L_j|>\lambda\sqrt{ V_n}/(2b)\right)+\mathbb P\left(V_n>4b^2 n\right).$$
Let $\rho>\sqrt{2}$. 
Following \cite{Bolthausen}, we obtain that 
$$\mathbb P\left(\max_{j\le n}|L_j|\ge \sigma\rho \sqrt{V_n}\right) \le 
2\mathbb P\left(|L_n|>(\rho-\sqrt{2})\sigma\sqrt{V_n}\right),$$
with $\sigma:=\sqrt{\var(h_1(\xi_1))}$.
But we know that $(L_n/\sqrt{V_n})_n$ converges in distribution to $Y$, so 
\begin{eqnarray*}
\limsup_{n\rightarrow +\infty}\mathbb P\left(\max_{j\le n}|L_j|\ge \rho \sigma\sqrt{V_n}\right)
      \le  2\mathbb P\left(|Y|>(\rho-\sqrt{2})\sigma\right).
\end{eqnarray*}
Choose $\rho>\sqrt 2$ such that $2(2b\rho\sigma)^2 \mathbb P\left(|Y|>(\rho-\sqrt{2})\sigma\right)<
    \varepsilon$ (this is possible since $Y$ is gaussian). The tightness criteria
\eqref{eq:tightness} is satisfied with $\lambda=2b\rho\sigma$.

For the remainder term, we have proved in Proposition \ref{lem2} that
\begin{equation*}
\sup_{t\in[0,1]}\left|n^{-2+\frac{1}{\alpha'}-\frac{1}{\alpha'\beta}}R_{[nt]}\right|\longrightarrow 0
\end{equation*}
in probability. The statement of the theorem follows by Slutzky's theorem.

In the cases $\alpha\le 1$, note that the uniform convergence of the remainder $R_{[nt]}$ implies the convergence with respect to the $M_1$-topology.

\end{proof}

\begin{proof}[Proof of Theorem \ref{theo2}] We use again the Hoeffding decomposition
\begin{equation*}
\frac{U_n}{n^{\frac{7}{4}}(\log\log n)^{\frac{3}{4}}}=\frac{n-1}{n}\frac{L_n}{(n\log\log n)^{\frac{3}{4}}}+\frac{R_n}{n^{\frac{7}{4}}(\log\log n)^{\frac{3}{4}}}.
\end{equation*}
For the remainder term we use Proposition \ref{lem2} with $\alpha=\beta=2$:
\begin{equation*}
R_n=o(n^{2-\frac{1}{2}+\frac{1}{4}})=o(n^{\frac{7}{4}}(\log\log n)^{\frac{3}{4}})\quad\mbox{almost surely}.
\end{equation*}
As $L_n=\sum_{i=1}^nh_1(\xi(S_i))$, we can apply Corollary 1 of Zhang \cite{zhan} and we obtain
\begin{equation*}
\limsup_{n\rightarrow\infty}\pm\frac{L_n}{(n\log\log n)^{\frac{3}{4}}}=\frac{2^{\frac{1}{4}}\var(\xi)^{\frac{1}{2}}}{3\var(X)^{\frac{1}{4}}}
\end{equation*}
almost surely, which leads to the statement of the theorem.
\end{proof}

\section*{Acknowledgement}
The research was supported by the DFG Sonderforschungsbereich 823 (Collaborative Research Center) {\em Statistik nichtlinearer dynamischer Prozesse}.

\end{document}